\documentclass[12pt]{article}
  
\usepackage{epsfig}
\usepackage{latexsym}
\usepackage{caption}
\usepackage{subcaption}
\usepackage{amsfonts}
\usepackage{authblk}

\usepackage{amssymb}
\usepackage{mathrsfs}
\usepackage{amsmath}
\usepackage{enumerate} 
\usepackage{enumitem}
\usepackage{graphicx}
\usepackage{MnSymbol}
\usepackage{float}
\usepackage{pict2e} 
\usepackage{graphics}
\usepackage{tikz}
\usepackage{cite}
\usepackage{ulem}
\usetikzlibrary{calc}

\usepackage[title]{appendix}

\numberwithin{equation}{section}

\usepackage{soul,color} 
\definecolor{lightgreen}{RGB}{153,255,153}

\definecolor{brightred}{RGB}{255,80,102}
 
\definecolor{lightblue}{RGB}{147,224,255} 

\soulregister\cite7
\soulregister\ref7
\soulregister\pageref7

\usepackage[colorinlistoftodos,prependcaption,textsize=scriptsize,color=olive!70, textwidth=30mm,]{todonotes}

\usepackage[pdfauthor={derajan},pdftitle={How to do this},pdfstartview=XYZ,bookmarks=true,
colorlinks=true,linkcolor=blue,urlcolor=blue,citecolor=blue,pdftex,bookmarks=true,linktocpage=true,  hyperindex=true]{hyperref}

\usepackage{color}

\makeatletter

\usetikzlibrary{patterns}

\tikzset{
	pics/mynodeA/.style={
		code={
			\draw[line width=1 pt] (0,0) circle(0.3cm);
			\foreach \i in {1,...,4}
			\fill (\i*90-45:0.2cm) coordinate (n\i) circle(1 pt)
			\ifnum \i>1 
			foreach \j in {\i,...,1}{(n\i) edge (n\j)} 
			\fi;
		}
	},pics/mynodeB/.default=1,
	Tnode/.style={circle,path picture={
			\path let
			\p1 = (path picture bounding box.south west),
			\p2 = (path picture bounding box.north east),
			\n1 = {scalar(veclen(\x2-\x1,\y2-\y1)/1cm/sqrt(2))}
			in (path picture bounding box.center)
			pic{mynodeA=\n1};
	}},
}

\tikzset{
	pics/mynodeB/.style={
		code={
			\draw[line width=1 pt] (0,0) circle(0.35cm);
			\foreach \i in {1,...,6}
			\fill (\i*60:0.25cm) coordinate (n\i) circle(1 pt)
			\ifnum \i>1 
			foreach \j in {\i,...,1}{(n\i) edge (n\j)} 
			\fi;
		}
	},pics/mynodeB/.default=1,
	Fnode/.style={circle,path picture={
			\path let
			\p1 = (path picture bounding box.south west),
			\p2 = (path picture bounding box.north east),
			\n1 = {scalar(veclen(\x2-\x1,\y2-\y1)/1cm/sqrt(2))}
			in (path picture bounding box.center)
			pic{mynodeB=\n1};
	}},
}

\makeatother

\begin{document}
	
	\newtheorem{theorem}{Theorem}[section]
	\newtheorem{observation}[theorem]{Observation}
	\newtheorem{corollary}[theorem]{Corollary}
	\newtheorem{algorithm}[theorem]{Algorithm}
	\newtheorem{definition}[theorem]{Definition}
	\newtheorem{guess}[theorem]{Conjecture}
	\newtheorem{claim}{Claim}[section]
	\newtheorem{problem}[theorem]{Problem}
	\newtheorem{question}[theorem]{Question}
	\newtheorem{lemma}[theorem]{Lemma}
	\newtheorem{proposition}[theorem]{Proposition}
	\newtheorem{fact}[theorem]{Fact}

	\captionsetup[figure]{labelfont={bf},name={Fig.},labelsep=period}
	
	\makeatletter
	\newcommand\figcaption{\def\@captype{figure}\caption}
	\newcommand\tabcaption{\def\@captype{table}\caption}
	\makeatother

	\newtheorem{acknowledgement}[theorem]{Acknowledgement}
	
	\newtheorem{axiom}[theorem]{Axiom}
	\newtheorem{case}[theorem]{Case}
	\newtheorem{conclusion}[theorem]{Conclusion}
	\newtheorem{condition}[theorem]{Condition}
	\newtheorem{conjecture}[theorem]{Conjecture}
	\newtheorem{criterion}[theorem]{Criterion}
	\newtheorem{example}[theorem]{Example}
	\newtheorem{exercise}[theorem]{Exercise}
	\newtheorem{notation}{Notation}
	\newtheorem{solution}[theorem]{Solution}
	\newtheorem{summary}[theorem]{Summary}

	\newenvironment{proof}{\noindent {\bf
			Proof.}}{\rule{3mm}{3mm}\par\medskip}
	\newcommand{\remark}{\medskip\par\noindent {\bf Remark.~~}}
	\newcommand{\pp}{{\it p.}}
	\newcommand{\de}{\em} 
	\newcommand{\mad}{\rm mad} 
	\newcommand{\wrt}{with respect to } 
	
	\newcommand{\qed}{\hfill\rule{0.5em}{0.809em}}
	
	\newcommand{\var}{\vartriangle}

	\title{\Large \bf The strong fractional choice number and the strong fractional paint number of graphs}

	\author[1,2]{Rongxing Xu \thanks{E-mail: xurongxing@zjnu.edu.cn. Grant Numbers: NSFC 11871439. Supported also by Fujian Provincial Department of Science and Technology(2020J01268).}}
	\author[1]{Xuding Zhu \thanks{E-mail: xdzhu@zjnu.edu.cn. Grant Number: NSFC 11971438,12026248, U20A2068. }}
	
	\affil[1]{\small Department of Mathematics, Zhejiang Normal University, Jinhua, Zhejiang, 321000, China.}
	  
	\affil[2]{\small School of Mathematical Sciences, University of Science and Technology of China, Hefei, Anhui, 230026, China}

	\maketitle

	\begin{abstract}
		This paper studies the strong fractional choice number  $ch^s_f(G)$  and the strong fractional paint number  $\chi^s_{f,P}(G)$  of a graph $G$. We prove that these parameters of any finite graph are rational numbers. On the other hand,    for any  positive integers $p,q$   satisfying $2 \le \frac{2p}{2q+1} \leq \lfloor\frac{p}{q}\rfloor$, there exists a graph $G$ with $ch^s_f(G) = \chi^s_{f,P}(G) = \frac{p}{q}$. The relationship between $\chi^s_{f,P}(G)$ and $ch^s_f(G)$ is explored. We prove  that the gap $\chi^s_{f,P}(G)-ch^s_f(G)$ can be arbitrarily large. 
	The strong fractional choice number of a family $\mathcal{G}$ of graphs is the supremum of the strong fractional choice number of   graphs in $\mathcal{G}$.  
	Let $\mathcal{P}$ denote the class of planar graphs and $\mathcal{P}_{k_1,\ldots, k_q}$ denote the class of planar graphs without $k_i$-cycles for $i=1,\ldots, q$. We   prove that  $3 + \frac{1}{2} \leq ch^s_f(\mathcal{P}_{  4}) \leq 4$, $ch^s_f(\mathcal{P}_{  k})=4$ for $k \in \{5,6\}$, $3 +\frac{1}{12} \leq ch^s_f(\mathcal{P}_{  4,5}) \leq 4$ and $ch^s_f(\mathcal{P}) \ge 4+\frac 13$. The last result improves the lower bound $4+\frac 29$ in [Zhu, multiple list colouring of planar graphs, Journal of Combin. Th. Ser. B,122(2017),794-799].     
	\end{abstract}

	\section{Introduction}
	
	Suppose $G$ is a graph, $f$ and $g$ are two functions from $V(G)$ to $\mathbb{N}$, with $f(v) \geq g(v)$ for every $v \in V(G)$. An \emph{$f$-assignment} of $G$ is a mapping $L$ which assigns to each vertex
	$v$ of $G$ a set $L(v)$ of $f(v)$ integers as {\em permissible colours}. A {\em $g$-fold colouring} of $G$ is a mapping $S$ which
	assigns to each vertex $v$ of $G$ a set $S(v)$ of $g(v)$ colours
	such that for any two adjacent vertices $u$ and $v$, $S(u) \cap S(v)
	= \emptyset$. 
	
	Given a list assignment $L$ of $G$, an
	$(L,g)$-colouring of $G$ is a $g$-fold colouring $S$ of $G$ such that for
	each $v$, $S(v) \subseteq L(v)$. We say $G$ is \emph{$(L,g)$-colourable} if there exists an $(L, g)$-colouring of $G$. For a positive integer $a$, we write $f \equiv a$ if $f$ is the constant function with $f(v)=a$ for every vertex $v$. If $g \equiv b$, then $(L, g)$-colourable is denoted by $(L, b)$-colourable. 
	If $L(v)=\{1,2,\ldots,a\}$ for each $v \in V(G)$, then $(L, b)$-colourable is called  $(a,b)$-colourable.  The \emph{$b$-fold chromatic number} $\chi_b(G)$ of $G$ is the least $k$ such that $G$ is $(k,b)$-colourable. The $1$-fold chromatic number of $G$ is also called the chromatic number of $G$ and denoted by $\chi(G)$. The {\em fractional chromatic number}   of   $G$ is  defined as $	\chi_f(G)  = \inf \{\frac ab: \text{$G$ is $(a,b)$-colourable}\}$.

	Similarly, we say $G$ is {\em $(f,g)$-choosable} if for every $f$-list assignment $L$,
	$G$ is $(L, g)$-colourable.
	\begin{itemize}
		\item If  $f \equiv a$ and $g \equiv b$, then $(f,g)$-choosable  is called $(a, b)$-choosable.
		\item If $b=1$, then $(f,1)$-choosable is called $f$-choosable.
		\item $(a,1)$-choosable is also called $a$-choosable.
	\end{itemize}
	The {\em choice number} $ch(G)$ of $G$ is the minimum $k$ such that
	$G$ is $k$-choosable. The {\em $b$-fold choice number}  $ch_b(G)$ of $G$
	is the minimum $k$ such that $G$ is $(k,b)$-choosable. The {\em fractional choice number}   of   $G$ is defined as $ch_f(G)   = \inf \{\frac ab: \text{$G$ is $(a,b)$-choosable}\}$.
	
	List colouring of graphs was introduced in the 1970s by Vizing
	\cite{Vizing1976} and independently by Erd\H{o}s, Rubin and Taylor
	\cite{ERT1979}. The subject offers a large number of challenging
	problems and has attracted an increasing attention since 1990.
	Readers are referred to \cite{TuzaSurvey1997} for a comprehensive survey
	on results and open problems.  
	
	The paint number of a graph is a variation of the choice number of a graph.
	Given two functions $f$ and $g$ from $V(G)$ to $N$, with $f(v) \ge g(v)$ for all $v \in V(G)$, the {\em $(f,g)$-painting game on $G$} is played by two players: Lister and Painter. Initially, each vertex $v$ is given $f(v)$ tokens and is uncolourred.  On each round, Lister selects a set $U$ of  vertices and removes one token from each chosen vertex.
	Painter chooses an independent subset $I$ of $U$ and colours each vertex of $I$ with a new colour.
	If at the end of some round, there is a vertex $v$ with no tokens left and coloured with less than $g(v)$ colours, then Lister wins the game.
	If at the end of some round, each vertex $v$ is coloured with $g(v)$ colours, then Painter wins the game.
	We say $G$ is {\em $(f,g)$-paintable } if Painter has a winning strategy for the $(f,g)$-painting game. 
	\begin{itemize}
		\item If $f \equiv a$ and $g \equiv b$,, then $(f,g)$-paintable  is called $(a, b)$-paintable.
		\item If $b=1$, then $(f,b)$-paintable is called $f$-paintable.
		\item $(a,1)$-paintable is also called $a$-paintable.
	\end{itemize} 
	
	The  \emph{$b$-fold paint number} $\chi_{b, P}(G)$ is the minimum $k$ such that $G$ is $(k,b)$-paintable, and $\chi_{1, P}(G)$ is called the {\em paint number} (or the {\em online choice number}) of $G$, and denoted by $\chi_P(G)$.    The {\em fractional paint number}   of  $G$ is defined as $\chi_{f,P}(G)=\inf \{\frac ab: \text{$G$ is $(a,b)$-paintable}\}$.   
	
	It follows from the definition that for any graph $G$, $\chi_f(G) \le ch_f(G) \le \chi_{f,P}(G)$. 
	It was proved in \cite{ATV1997} that $\chi_f(G) = ch_f(G) $ for every graph $G$, and proved in \cite{Gutowski2011} that $\chi_f(G) = \chi_{f,P}(G) $ for every graph $G$. So the fractional chromatic number, the fractional choice number and the fractional paint number of a graph are a same invariant.  As a variation of the fractional choice number, the concept of strong fractional choice number of a graph was introduced in \cite{Zhu2017}. 
	\begin{definition}
		\label{def-strong}
		Assume $G$ is a graph and $r$ is a real number. We say $G$ is {\em strongly fractional $r$-choosable} (respectively,   {\em strongly fractional $r$-paintable} or {\em strongly fractional $r$-colourable} ) if $G$ is $(a,b)$-choosable (respectively,  $(a,b)$-paintable, or $(a,b)$-colourable) for any $(a,b)$ for which $\frac ab \ge r$. The {\em strong fractional choice number}  of   $G$
		is defined as $$ch^s_f(G)= \inf \{r \in \mathbf{R}:  G \ \text{is strongly fractional $r$-choosable}\}.$$
		The {\em strong fractional paint number}  of   $G$
		is defined as $$\chi^s_{f,P}(G)= \inf \{r \in \mathbf{R}:  G \ \text{  is strongly fractional $r$-paintable}\}.$$
		We also define the {\em strong fractional chromatic number}  of   $G$ as 
		$$\chi^s_f(G)= \inf \{r \in \mathbf{R}:  G \ \text{is strongly fractional $r$-colourable}\}.$$
		The strong fractional choice number, the strong fractional paint number and the strong fractional chromatic number of a class $\mathcal{G}$ of graphs is defined as 
		$$ch^s_f(\mathcal{G})=\sup\{ch^s_{f}(G):G \in \mathcal{G}\}, \  \chi_{f,P}^s(\mathcal{G})=\sup\{\chi_{f,P}^s(G):G \in \mathcal{G}\}, \  \chi_{f}^s(\mathcal{G})=\sup\{\chi_{f}^s(G):G \in \mathcal{G}\}.$$
	\end{definition}
	
	The paper studies basic properties of these parameters, and upper and lower bounds for these parameters for special families of graphs. 
	In Section \ref{basic-properties}, we show that both $ch^s_f(G)$ and $\chi^s_{f,P}(G)$ are rational numbers. 
	In Section \ref{rational}, we study the problem as  which  rational numbers are  the strong fractional choice number and strong fractional paint number of graphs.   We conjecture that for every rational $ r \ge 2$,   there exists a graph $G$ with $ch^s_f(G)=r$ and a graph $G$ with $\chi^s_{f,P}(G)=r$, and prove that  for any  positive integers $p,q$, where $p \geq  2q$  satisfying $\frac{2p}{2q+1} \leq \lfloor\frac{p}{q}\rfloor$, there exists a graph $G$ with $ch^s_f(G) = \chi^s_{f,P}(G) = \frac{p}{q}$.
	In Section \ref{sec:uppper-bound}, we show that the gap $\chi^s_{f,P}(G)-ch^s_f(G)$ can be arbitrarily large.  
	In Section \ref{section-planar-lower-bound},  we study upper and lower bounds for the strong fractional choice number of planar graphs. Let ${\cal P}$ denote the family of planar graphs and for positive integers $k_1, k_2, \ldots, k_q$, let ${\cal P}_{k_1, \ldots,k_q}$ be the family of planar graphs without cycles of lengths $k_i$ for $i=1,\ldots, q$.  It was proved in \cite{Zhu2017} that $ch_f^s({\cal P}) \ge 4+\frac 29$. We improve this result and prove that $ch_f^s({\cal P}) \ge 4+\frac 13$.   It is also shown that   $3 + \frac{1}{2} \leq ch^s_f(\mathcal{P}_{4}) \leq 4$, $ch^s_f(\mathcal{P}_{k})=4$ for $k \in \{5,6\}$, and $3 +\frac{1}{12} \leq ch^s_f(\mathcal{P}_{4,5}) \leq 4$. 
	Some open problems are posed in Section \ref{open problem}.

	\section{Basic Properties} 
	\label{basic-properties}
	  
Lemma \ref{lem-alternatedef} gives an alternative definitions of $ch^s_f(G)$ and $\chi^s_{f,P}(G)$.

	\begin{lemma}
		\label{lem-alternatedef}
		For any graph $G$,
		$$ch^s_f(G)= \sup\{ \frac{ch_k(G)-1}{k}: k \in \mathbf{N}\}, \ \chi^s_{f,P}(G)= \sup\{ \frac{\chi_{k, P}(G)-1}{k}: k \in \mathbf{N}\}.$$
	\end{lemma}
	
	\begin{proof}
		Let $r = \sup\{ \frac{ch_k(G)-1}{k}: k \in \mathbf{N}\}$.
		Then for any $\epsilon > 0$, there is an integer $k$ such that $(r-\epsilon)k <
		ch_k(G)-1$. Thus $\lceil (r-\epsilon)k \rceil < ch_k(G)$ and  $G$ is not $\lceil (r-\epsilon)k \rceil, k)$-choosable. Therefore, $ch^s_f(G) \ge r-\epsilon$ 
		for any $\epsilon > 0$, which implies that
		$ch^s_f(G) \ge r$. On the other hand, for any $\epsilon > 0$, for any integer $k$,
		$\lceil (r+\epsilon) k \rceil \ge ch_k(G)$. Hence $G$ is $(\lceil (r+\epsilon) k \rceil,
		k)$-choosable. So $ch^s_f(G) \le r+\epsilon$ for any $\epsilon > 0$, which implies that
		$ch^s_f(G) \le r$. Therefore $ch^s_f(G)=r$. The other part of Lemma \ref{lem-alternatedef} is
		proved similarly.
	\end{proof}

	The following lemma was proved in \cite{Gutowski2011}. For the completeness of this paper, we present a
	short proof.
	
	\begin{lemma}
		\label{lem-fracchi}
		Assume $G$ is a finite graph. Then for any $\epsilon > 0$, there is a constant $k_0$ such that
		for any $k\ge k_0$,
		$\frac{ch_k(G)}{k} \le  \frac{\chi_{k, P}(G)}{k} \le \chi_f(G)+\epsilon$.
	\end{lemma}
	\begin{proof}
		Assume $\chi_f(G)=a/b$ and $\phi$ is a $b$-fold colouring of $G$ using colours
		$\{1,2,\ldots, a\}$ ($a,b$ need not be coprime).
		Assume $k > \frac{a2^{|V(G)}|}{\epsilon}$  and let $m= k(\frac{a}{b}+\epsilon)$ (for simplicity, we may choose $\epsilon$ so that $k(\frac{a}{b}+\epsilon)$ is an integer). It suffices to show that Painter has a winning
		strategy for
		the $(m,k)$-painting game on $G$. 
		
		For $i=1,2, \ldots$, assume the set   chosen by
		Lister at the $i$th round is $U_i$.
		Let  $$t_i=|\{j \le i: U_j=U_i\}|,$$ and let $\tau_i \in \{1,2, \ldots, a\}$ be the unique integer for which 
		$\tau_i \cong t_i \mod{a}$.  
		Painter's strategy is to colour all
		the vertices in the set $I_i=\phi^{-1}(\tau_i)\cap U_i$ in the $i$th round. As $\phi^{-1}(\tau_i)$ is an independent set, $I_i$ is an independent set. 
		
		We shall show that this is a winning strategy for Painter, i.e., when the game ends, every vertex will be coloured by at least $k$ colours.  
		
		Assume to the contrary that at the end of some round, say at the end of the $i$th round,
		a vertex $v$ has no token left (hence $v$ has been chosen $m = k(\frac{a}{b}+\epsilon))$ times by
		Lister)
		and is coloured in   $k(v) < k$ rounds.
		
		For each subset $U$ of $V(G)$ and for each  $t \in \{1,2,\ldots, a\}$,  let $$(U,t) = \{j \le i: U_j=U,  \tau_j=t\}.$$
		
		By the strategy, for each $j \le i$, if $j \in (U,t)$, $v \in U$ and $t \in \phi(v)$, then $v$
		is coloured in round $j$.
		Therefore,  $$k(v) = \sum_{ v \in U, t \in \phi(v)}|(U,t)|.$$
		
		For each subset $U$ of $V(G)$, let $t_U = |\{j \le i: U_j=U\}|$. It follows from the choice of
		colour $\tau_j$ that
		for any colour $t$, either $|(U,t)|=\lfloor \frac{t_U}{a} \rfloor$ or $ |(U,t)|= \lceil \frac{t_U}{a} \rceil$. Therefore 
		$$|(U,t)| \ge   \frac{t_U}{a} -1.$$
		
		Note that $\sum_{v \in U}t_U=m$ is the total number of times vertex $v$ is chosen by Lister.
		Since $\phi(v)$ is a $b$-subset of $\{1,2,\ldots,a\}$, we conclude that
		$$k(v)=\sum_{ v \in U, t \in \phi(v)}|(U,t)| \ge b \sum_{ v \in U} \left(\frac{t_U}{a}-1\right)\geq \frac{bm}{a}-b2^{|V(G)|}=\frac{k(a/b+\epsilon)b}{a}-b2^{|V(G)|} \ge k,$$ a contradiction.
	\end{proof}

	\begin{theorem}
		\label{thm-rational}
		For any finite graph $G$, $ch^s_f(G)$ and $\chi^s_{f,P}(G)$ are rational numbers.
	\end{theorem}
	
	\begin{proof}
		If  $\frac{\chi_{k,P}(G)-1}{k} \le \chi_f(G)$ for every positive integer $k$, then
		it follows from Lemma \ref{lem-alternatedef}
		that $\chi^s_{f,P}(G) \le \chi_f(G)$. Since
		$\chi_f(G) \le \chi^s_{f,P}(G)$, we conclude that
		$\chi^s_{f,P}(G)=\chi_f(G)$, which is  a rational number. 
		
		Assume there is an integer $k_0$ such that
		$\frac{\chi_{k_0,P}(G)-1}{k_0} > \chi_f(G)$. Let $\epsilon = \frac{\chi_{k_0,P}(G)-1}{k_0} - \chi_f(G) > 0$.
		By Lemma \ref{lem-fracchi}, there is a constant $k_1 \ge k_0$ such that for $k \ge k_1$,
		$\frac{\chi_{k,P}(G)}{k} \le \chi_f(G)+\epsilon$.
		Hence
		$$\sup\{ \frac{\chi_{k, p}(G)-1}{k}: k \in \mathbf{N}, k  \ge k_1\} \le
		\frac{\chi_{k_0,P}(G)-1}{k_0}.$$
		Therefore $$\chi_{f,P}^s(G) = \sup\{ \frac{\chi_{k,P}(G)-1}{k}: k \in \mathbf{N}\} = \max
		\{ \frac{\chi_{k,P}(G)-1}{k}: 1 \le k \le k_1\}$$
		is a rational number. Moreover, the supremum in Lemma \ref{lem-alternatedef} is attained.
		
		The part of the lemma concerning $ch^s_f(G)$ is proved similarly.
	\end{proof}
	
	Lemma \ref{lem-alternatedef} gives an alternate definition of $ch_f^s(G)$ and $\chi_{f,P}^s(G)$. It follows from the proof of Theorem \ref{thm-rational} that either $ch_f^s(G)=\chi_f(G)$ or the supremum in the definition 
	$ch_f^s(G) = \sup\{ \frac{ch_k(G)-1}{k}: k \in \mathbf{N}\}$ is attained.
	However, the infimum in the definition $ch_f^s(G) = \inf \{r: \text{$G$ is strongly fractional $r$-choosable}\}$ may be not attained even if $ch_f^s(G) \ne \chi_f(G)$.
	Similarly,  either $\chi_{f,P}^s(G) = \chi_f(G)$ or the sumpremum in the definition $\chi_{f,P}^s(G) = \sup\{ \frac{\chi_{k,P}(G)-1}{k}: k \in \mathbf{N}\}$ is attained. But the infimum in the definition $\chi_{f,P}^s(G) = \inf \{r: \text{$G$ is strongly fractional $r$-paintable}\}$ may be not attained even if $\chi_{f,P}^s(G) \ne \chi_f(G)$.

	\section{Constructing graphs with given
		$ch^s_f(G)$ 
		and 
		$\chi^s_{f, P}(G)$}
	\label{rational}
	
	By Theorem \ref{thm-rational}, for any finite graph $G$,  $ch^s_{f}(G)$ and $\chi^{s}_{f,P}(G)$ are rational numbers. A natural question is whether every rational number $r \ge 2$ is the strong fractional choice (paint) numbers of a graph.  We conjecture that the answer is yes. In this section, for some rational numbers $p/q$, we construct graphs $G$ with $\chi_{f,P}^s(G) = ch_f^s(G) = p/q$.

	Given a graph $G$, a subset $S$ of $G$ and two graphs $H_1$ and $H_2$, let    $G[S:H_1, H_2]$ be the graph with  
	\begin{align*}
		V(G[S:H_1, H_2]) & = \{(u,v): u \in S \text{\ and \ } v \in V(H_1), \text{ or\ } u \in V(G) \setminus S \text{\ and \ } v \in V(H_2)\}, \\
		E(G[S:H_1, H_2]) & = \{(u,v)(u',v'): uu' \in E(G), \text{ \ or } u=u' \in S, vv' \in E(H_1)  \\
		& \text{ or }
		u=u' \in V(G) \setminus S, vv' \in E(H_2)\}.
	\end{align*}
	Note that if $S=V(G)$ or $H_1=H_2$, then  $G[S:H_1,H_2]=G[H_1]$ is the {\em lexicographic product} of $G$ and $H_1$. In the rest of this section, we let $G_{n,m,k}$ denote the graph $C_{2k+1}[I:K_n, K_m]$, where $I$ is  a maximum independent set of $C_{2k+1}$, see Fig.\ref{replacing-example} for the example of $G_{6,4,3}$.
	\begin{figure}[htbp]
		\centering
		\begin{tikzpicture}[>=latex,	
			roundnode/.style={circle, draw=black,fill= red, minimum size=1mm, inner sep=0pt}]  
			\node[Fnode, minimum size=0.7cm] (A1) at (90-720/7:1.5) {};
			\node[Tnode, minimum size=0.6cm] (A2) at (90-360/7:1.5) {};
			\node[Fnode,minimum size=0.7cm] (A3) at (90:1.5) {};
			\node[Tnode,minimum size=0.6cm] (A4) at (90+360/7:1.5) {};
			\node[Fnode,minimum size=0.7cm] (A5) at (90+720/7:1.5) {};
			\node[Tnode,minimum size=0.6cm] (A6) at (90+1080/7:1.5) {};
			\node[Tnode,minimum size=0.6cm] (A7) at (90-1080/7:1.5) {};
			\draw[line width = 2pt] (A1)--(A2)--(A3)--(A4)--(A5)--(A6)--(A7)--(A1);
		\end{tikzpicture}
		\caption{$G_{6,4,3}$}
		\label{replacing-example}
	\end{figure}
	
	\begin{theorem}
		\label{thm:n+m+m/q}
		For any positive integer $n, m, k$ with $n \geq m$,  
		$$ch_{f}^s(G_{n,m,k})=\chi^s_{f,P}(G_{n,m,k})=\chi_f(G_{n,m,k})=n+m+\frac{m}{k}.$$
	\end{theorem}
	
	\begin{proof}	
		Assume the vertices of $C_{2k+1}$ are $(v_0,v_1,\ldots v_{2k})$ in this cyclic order, and assume that $I= \{v_1, v_3, \ldots, v_{2k-1}\}$ and  $G_{n,m,k} = C_{2k+1}[I: K_n, K_m]$. 
		
		For   $s \in \{0,1,\ldots, 2k\}$, let $$V_s=\{(x,y) \in V(G_{n,m,k}): x=v_s\}.$$ For any vertex set $S\subseteq V(G_{n,m,k})$, let $$\partial(S) = \{s: S \cap V_s \neq \emptyset\}.$$
		
		It is clear  that $\alpha(G_{n,m,k})= \alpha(C_{2k+1})=k$. It is well known that 
		$\chi_f(G)\ge \frac{|V(G)|}{\alpha(G)}$ for any graph $G$, so we have   
		$$\chi_f(G_{n,m,k})\ge \frac{nk+m(k+1)}{k}=n+m+\frac{m}{k}.$$
		
		Since $\chi_f(G) \le ch^s_f(G) \le \chi^s_{f,p}(G)$ for any graph $G$, it suffices to show that $\chi^s_{f,P}(G_{n,m,k}) \le n+m+\frac{m}{k}$. For this purpose, we will show that for any $\frac{a}{b} \geq n+m+\frac{m}{k}$, $G_{n,m,k}$ is $(a,b)$-paintable.  In the following, we present a winning strategy for Painter in the $(a,b)$-painting game on $G_{n,m,k}$.
		
		For simplicity, we assume that if a vertex $v$ has been coloured with $b$ colours, then Lister
		will not choose $v$ in later moves.
		
		For $i=1,2, \cdots $, we denote by $U_i$ the set of vertices chosen by Lister, and by $X_i$ the independent set contained in $U_i$, coloured by Painter at the $i$th round. 
		
		Assume Lister has chosen $U_i$. We describe a strategy for Painter to choose the independent set $X_i$.
		
		We consider two cases.
		
		\medskip
		\noindent
		{\bf Case 1} $\partial(U_i) = \{0,1,2,\cdots, 2k\}$.
		
		Let $t_i = |\{j: j \le i \ \text{and } \partial(U_j) = \{0,1,\cdots, 2k\} \}|$. Let $\tau_i$ be the unique integer in $\{0, 1, \cdots, 2k\}$ which is congruent to $t_i$ modulo $2k+1$.
		Then $X_i$ is any independent set contained in $U_i$ with $\partial(X_i)= \{\tau_i, \tau_i+2,
		\cdots,	\tau_i+2k-2 \}$ (the summation in the indices are modulo $2k+1$).
		
		\medskip
		\noindent
		{\bf Case 2} $\partial(U_i) \neq \{0,1,\cdots, 2k\}$.

		Painter traverse the  sets $V_0, V_1, \ldots, V_{2k}$ of $G_{n,m,k}$ one by one in cyclic order  along the clockwise
		direction, starting at an arbitrary  set $V_s$ for which $s \notin \partial(U_i)$, and choose an independent set $X_i$ as follows: Initially, $X_i = \emptyset$ and vertices will be added to $X_i$ in the process.
		When we traverse to $V_j$, if $U_i \cap V_j \neq
		\emptyset$ and $X_i \cap V_{j-1} = \emptyset$,  then add a vertex from $V_j\cap U_i$ to $X_i$. Otherwise, do not add any vertex from $V_j$ into $X_i$
		(again the calculation in the indices are modulo $2k+1$).
		
		It follows from the definition  that in both cases, the set $X_i$ is an independent set of $G$ contained in $U_i$. We shall show that this is a winning strategy for Painter. First we have the following claim.
		
		\begin{claim}
			\label{giveup}
			Case 1 happens at most $\frac{mb(2k+1)}{k}$ times. 
		\end{claim}
		\begin{proof}
			Every $2k+1$ times Case 1 happens,  $k$ vertices (not necessarily distinct) in $V_0$ will be coloured. However,
			all the vertices from $V_0$ need to be coloured $mb$ times in total. Thus the claim holds.
		\end{proof}
		
		Assume to the contrary that at the end of some round, say at the end of the $r$th round,
		a vertex $v \in V_j$ has no token left and is coloured at most $b-1$ times.
		
		Let $B_r(v)=\{i \leq r: v \in U_{r} - X_{r}\}$, which is the collection of rounds of the game that $v$ is chosen by Lister  but not coloured by Painter at the end of $r$th round.
		Then $|B_r(v)| \geq a-b+1$. It follows from the strategy that for each $i \in B_r(v)$, one of the following   holds:
		
		\begin{itemize}
			\item   $V_{j-1} \cap X_i \neq \emptyset$, i.e.,  some vertex in $V_{j-1}$ is 
			coloured in this round.
			\item $\partial(U_{i}) = \{0,1,\cdots, 2k\}$ and $\tau_{i}=j+1$.  
			\item $V_{j} \cap X_{i} \neq \emptyset$ and $v \notin X_{i}$, i.e., some another vertex from $  V_j$ is cloured at this round.
		\end{itemize}
		
		Since    $|V_{j-1}\cup (V_{j}\setminus v)| \le m+n-1$ and each vertex in $V_{j-1}\cup (V_{j}\setminus v)$ is coloured at most $b$ times, we conclude that
		\begin{align}
			\label{eq-lex1}
			|\{i \leq r: \partial(U_{i}) & = \{0,1,\cdots, 2k\}, \tau_i= j+1| \nonumber  \\
			& \geq a-b+1-(m+n-1)b  \nonumber  \\
			& = a-(m+n)b+1. 
		\end{align}
		It follows from the definition of $\tau_{i}$ that 
		\begin{align}
			\label{eq-lex2}
			|\{i \leq r: \partial(U_{i}) = \{0,1,\cdots, 2k\}, \tau_{i}=j+1| \leq \lceil \frac{t_i}{2k+1} \rceil. 
		\end{align}
		By Claim \ref{giveup}, $t_i \leq \frac{(2k+1)mb}{k}$. Hence
		$$ \lceil \frac{t_{i}}{2k+1} \rceil  \leq \lceil \frac{(2k+1)mb}{k(2k+1)} \rceil = \lceil \frac{mb}{k} \rceil,$$ 
		Combining the inequality above   with Inequalities (\ref{eq-lex1}) and (\ref{eq-lex2}), we have
		$$\frac{mb}{k} > a-(m+n)b,$$ 
		that is  $$\frac{a}{b} < m+n+\frac{m}{k},$$ contrary to our assumption. 
	\end{proof}
	
	By setting $m=n=1$ in Theorem \ref{thm:n+m+m/q}, we have 
	\begin{corollary}
		\label{odd-cycle}
		$\chi_{f}(C_{2q+1})=ch^s_{f}(C_{2q+1})=\chi^s_{f,P}(C_{2q+1})=2+\frac{1}{q}$.
	\end{corollary}
	
	\begin{proposition}
		\label{n+m}
		For any  positive integers $p, q$ with $p \geq 2q$ and $\frac{2p}{2q+1} \leq \lfloor \frac{p}{q}\rfloor$, there exists a graph $G$ such that  
		$$ch_{f}^s(G)=\chi^{s}_{f,p}(G)=\chi_f(G)=\frac{p}{q}.$$
	\end{proposition}
	\begin{proof}
		Let $a=\lfloor\frac{p}{q}\rfloor(q+1)-p$ and $b=p- q\lfloor\frac{p}{q}\rfloor$. As $p\geq 2q$, we first have that $a,b >0$. On the other hand, note that by the condition $\frac{2p}{2q+1} \leq \lfloor \frac{p}{q}\rfloor$, we have that
		\begin{align*} 
			a-b = (2q+1)\lfloor\frac{p}{q}\rfloor-2p \geq 0. 
		\end{align*}
		So by setting $a=n$ and $b=m$ in Theorem \ref{thm:n+m+m/q}, we have that 
		$$ch^s_{f}(G_{a,b,q})= \chi^s_{f,P}(G_{a,b,q})=a+b+\frac{b}{q}=\frac{p}{q}.$$
		
		This completes the proof.
	\end{proof}
	
	According to Proposition \ref{n+m}, if $q \leq2$, then for any $p \geq 2q$,  there exists a graph $G$ with $ch_f^s(G)=ch_{f,P}^s(G) = p/q$. If $q=3$, the only cases $p \geq 2q$ unknown are $p=8$ and $p=11$.

	\section{Relation among $\chi_f^s(G)$, $ch_f^s(G)$ and $   \chi_{f,P}(G)$}
	\label{sec:uppper-bound}
	
	It follows from the definitions that for any graph $G$,
	$$\chi_f^s(G) \le ch_f^s(G) \le \chi_{f,P}^s(G).$$
	The gap  $ch_f^s(G)-\chi_f^s(G)$ can be arbitrarily large, as for complete bipartite graphs, we have $\chi_f^s(K_{n,n}) =2$ and $ch_f^s(K_{n,n}) \ge ch(K_{n,n})-1 \ge \log_2 n -(2+o(1)) \log_2 \log_2 n $.
	
	In the following we show that the difference $\chi_{f,P}^s(K_{n,n}) -ch_f^s(K_{n,n})$  also goes to infinity with $n$. It was proved in \cite{DGK2016} that for $n \geq 2^{k+3}$, the graph $K_{n,n}$ is not $k$-paintable, so  $\chi_P(K_{n,n}) \ge \log_2 n - 4$. Therefore,  
	$\chi_{f,P}^s(K_{n,n}) \ge \chi_P(K_{n,n})-1 \geq  \log_2 n - 5$.
	So it suffices to show that $\log_2 n - \chi_f^s(K_{n,n})$ goes to infinity with $n$.
	
	A \emph{$k$-uniform hypergraph} $H=(V,E)$ consists of a vertex set $V$ and an edge set   $E$, where each edge $e \in E$  is a $k$-subset of $V$.  A proper $c$-colouring of $H$ is a mapping $\phi: V \mapsto \{1,2,\cdots, c\}$ such that no edge is monochromatic. Hypergraph $2$-colourability, which is an alternate formulation of list colouring of complete bipartite graphs, is a central problem in combinatorics that has been studied in many papers (see \cite{Erdos1963,Beck1978,RS2000}, etc.)  Corresponding to $b$-fold list colouring of complete bipartite graphs, 
	we  define  a  \emph{$b$-proper $2$-colouring} of a hypergraph $H$ as a mapping $\phi: V(H) \to \{1,2\}$ such that for each edge $e$ of $H$, for each $i \in \{1,2\}$, $|\phi^{-1}(i) \cap e| \ge b$, i.e.,   each edge contains at least $b$ vertices of each colour.   We say that \emph{$H$ is $b$-proper $2$-colourable} if $H$ has a  $b$-proper $2$-colouring. Let $m(k,b)$ denote the minimum possible number of edges of a $kb$-uniform hypergraph which is not $b$-proper $2$-colourable.

	\begin{lemma}
		\label{lem-hyp}
		Every $p$-uniform hypergraph with $m$ edges satisfying $m \sum_{i=0}^{b-1}\binom{p}{i}\frac{1}{2^{p-1}} < 1$ has a $b$-proper $2$-colouring. As a result,
		$$m(k,b) \geq \left(\sum_{i=0}^{b-1}\binom{ks}{i}\right)^{-1}2^{kb-1}.$$
	\end{lemma}
	\begin{proof}
		Let $H=(V,E)$ be a $p$-uniform hypergraph satisfying the condition. 
		
		Colour the vertices of $H$ randomly by two colours with equal probability. We say an edge $e$ is  \emph{bad} if one colour is used on less than $b$ vertices in $e$. For each edge $e$, let $A_e$ be the event that $e$ is bad. Then 
		$${\rm Pr}(A_e)=2\sum_{i=0}^{b-1}\binom{p}{i}\frac{1}{2^{p}}=\sum_{i=0}^{b-1}\binom{p}{i}\frac{1}{2^{p-1}}.$$ 
		Therefore,  $${\rm Pr}(\bigvee_{e \in E} A_e) \le \sum_{e \in E}{\rm Pr}(A_e) = m{\rm Pr}(A_e) < 1.$$
		So there   exists a colouring such that there is no bad edges. 
	\end{proof}
	
	\begin{lemma}
		\label{lem:chsf-upper bound}
		Let $G$ be a bipartite graph with $n$ vertices. When $n$ is big enough, the $b$-fold choice number $ch_b(G)$ satisfies the following, $$\frac{ch_b(G)}{b} < \frac{1}{b}\log_{2}n+(1-\frac{1}{b})\log_2\log_2n+O(1).$$
	\end{lemma}
	
	\begin{proof}
		Let  $k=\frac{1}{b}\log_{2}n+(1-\frac{1}{b})\log_2\log_2n+C$, for some constant $C$, we shall prove that $G$ is $(kb,b)$-choosable for any $b \geq 2$.
		
		Assume $G=(X \cup Y, E)$ be a bipartite graph with $X$ and $Y$ being the two parts, and $L$ is a list assignment of $G$ with $|L(v)|=kb$ for each $v \in V(G)$. We construct a $kb$-uniform hypergraph $H$ with $V(H)=\bigcup_{v \in V(G)}L(v)$, and $E(H)=\{L(v):v \in V(G)\}$. So $|E(H)|=|V(G)|=n$. 
		
		Observe that if $H$ has a $b$-proper $2$-colouring, then $G$ is $(L,m)$-colourable. Indeed, each vertex is either labeled with red or blue in the $b$-proper $2$-colouring of $H$. Then for each vertex $v \in V(G)$, we can choose $b$ colours with label red for $v$ if $v \in X$, and choose $b$ colours with label blue for it if $v \in Y$. 
		
		Now, it suffices to prove that $H$ satisfies the condition in Lemma \ref{lem-hyp}, so we only need to verify that 
		$$n \sum_{i=0}^{b-1}\binom{kb}{i}\frac{1}{2^{kb-1}} < 1.$$ 
		The case that $b=1$ was proved in  \cite{Erdos1963}.  
		If $b=2$, then we have $kb=\log_2n+\log_2\log_2n+2C$, so,
		\begin{align*}
			n\sum_{i=0}^{b-1}\binom{kb}{i}\frac{1}{2^{kb-1}} &=  n(\log_2n+\log_2\log_2n+2C+1)\frac{1}{2^{\log_2n+\log_2\log_2n+2C-1}} \\
			& =   (\log_2n+\log_2\log_2n+2C+1)\frac{1}{2^{\log_2\log_2n+2C-1}} \\
			& =  \frac{2^t+t+2C+1}{2^{t+2C-1}},
		\end{align*}
		where $t=\log_2\log_2n$. When $n$ is large enough and hence $t$ is large enough, and $C \ge 1$, we have $n \sum_{i=0}^{b-1}\binom{kb}{i}\frac{1}{2^{kb-1}} < 1.$
		Similarly, we can verify the case for $b=3$.
		
		Assume $b\geq 4$. Using the inequality $\binom{n}{k}< (\frac{en}{k})^k$, we have 
		\begin{align*}
			n\sum_{i=0}^{b-1}\binom{kb}{i}\frac{1}{2^{kb-1}} & \le nb\left(\frac{e(\log_2n+(b-1)\log_2\log_2n+bC)}{b-1}\right)^{b-1}\frac{1}{2^{\log_2n+(b-1)\log_2\log_2n+bC-1}} \\
			& =   b\left(\frac{e(\log_2n+(b-1)\log_2\log_2n+bC)}{b-1}\right)^{b-1}\frac{1}{2^{(b-1)\log_2\log_2n+bC-1}} \\
			& =   \frac{b}{2^{bC-1}}\left(\frac{e(\log_2n+(b-1)\log_2\log_2n+bC)}{(b-1)\log_2n}\right)^{b-1}.
		\end{align*}
		Again when $n$ is large enough, we have $n \sum_{i=0}^{b-1}\binom{kb}{i}\frac{1}{2^{kb-1}} < 1.$
		This finishes the proof of the lemma.
	\end{proof}
	
	In 2000, Radhakrishnan and Srinivasan \cite{RS2000} actually gave a better lower bound for $m(k,1)$, who showed that $m(k,1)=\Omega(2^k\sqrt{\frac{k}{\ln k}})$.  
	This implies that $\frac{ch_1(G)}{1}=ch(G) \leq \log_2n - (\frac 12- o(1)) \log_2 \log_2 n$ if $G$ is a complete bipartite graph with $n$ vertices. Combing this fact and Lemma \ref{lem:chsf-upper bound} and Lemma \ref{lem-alternatedef}, we have the following proposition.
	
	\begin{corollary}
		Let $G$ be a bipartite graph with $n$ vertices. When $n$ is big enough, 
		$$ch^s_f(G) \leq \log_2n - (\frac 12- o(1)) \log_2 \log_2 n.$$ Consequently, $\chi^s_{f,P}(K_{n,n})-ch^s_f(K_{n,n})$ can be arbitrarily large.
	\end{corollary}

	Although the gaps  $ch_f^s(G)-\chi_f^s(G)$ and $\chi_{f,P}(G)-ch_f^s(G)$ can  be arbitrarily large, there are also many graphs $G$ for which the equality 
	$ch_f^s(G)=\chi_f^s(G)$ and/or $\chi_{f,P}(G)=ch_f^s(G)$ hold.  Recall that a graph $G$ is called {\em chromatic-choosable} if $\chi(G)=ch(G)$. The study of chromatic choosable graphs attracted a lot of attention. The well-known list colouring conjecture asserts that line graphs are chromatic-choosable.  
	This conjecture remains largely open, however, it was  shown by Galvin \cite{Galvin1995} the the line graphs of  bipartite graphs are chromatic-choosable. 
	This result extends to strong fractional choice number and strong fractional paint number.
	
	\begin{theorem}
		If $G=L(H)$ is the line graph of a bipartite graph $H$, then $$\chi_f^s(G) = ch_f^s(G) = \chi_{f,P}^s(G) = \Delta(H).$$ 
	\end{theorem}
	\begin{proof}
		It is well-known that $\chi(G) = \omega(G) =\Delta(H)$. Hence $\chi_f^s(G) \geq  \Delta(H)$. It remains to show that $\chi_{f, P}(G) \leq \Delta(H)$. 
		
		An orientation $D$ of $G$ is {\em kernel perfect} if   any subset $X$ of $V(D)$ contains an independent set $I$ such that for any $v \in X-I$, $N_D^+(v) \cap I \ne \emptyset$. Here $N_D^+(v)$ is the set of out-neighbours of $v$. We set $N_D^+[v] = N_D^+(v) \cup \{v\}$. 
		It was proved in \cite{Galvin1995} that $G$ has a kernel perfect orientation $D$ with $\Delta^+(D) = \Delta(H)$. 
		On the other hand, for any kernel perfect orientation $D$ of $G$, for any $f,g: V(D) \to \mathbb{N}$, if $f(v) \ge \sum_{u \in N_D^+[v]}g(u)$ for every vertex $v$, then it is easy to show by induction on $\sum_{v \in V(D)}f(v)$ that Painter has a winning strategy in the $(f,g)$-painting game.  
		
		Indeed, if Lister choose a subset $X$ of $V(G)$ in a round, then Painter chooses an independent set $I$ of $X$ for which $N_D^+(v) \cap I \ne \emptyset$ for all $v \in X-I$. Let 
		\[
		f'(v) = \begin{cases} f(v)-1, &\text{ if $v \in X-I$}, \cr
			f(v), &\text{otherwise.,}
		\end{cases}
		\]
		and 
		\[
		g'(v) = \begin{cases} g(v)-1, &\text{ if $v \in I$}, \cr
			g(v), &\text{otherwise.}
		\end{cases}
		\]
		
		It follows from the definition  
		that for any vertex $v$, we still have 
		$f'(v) \ge \sum_{u \in N_D^+[v]}'(u)$.
		By induction hypothesis, Painter has a winning strategy in the $(f',g')$-painting game on $G$. Therefore Painter has a winning strategy for the $(f,g)$-painting game on $G$. 
		
		For any positive integer $m$, by letting  $f(v) =m (d_D^+(v) +1) \le \Delta(H)m$ and $g(v) = m$ for each vertex $v$, we have that $G$ is $(\Delta(H)m, m)$-paintable. Hence $\chi_{f,P}^s(G) \le \Delta(H)$. 		 	
	\end{proof}
	
	\section{$ch^s_f(G)$ for planar graphs}
	\label{section-planar-lower-bound}
	
	In this section, we  study the strong fractional choice number of planar graphs. Let $\mathcal{P}$ denote the class of planar graphs and for integers $k_1,\ldots, k_q \ge 3$, let $\mathcal{P}_{k_1,\ldots, k_q}$ denote the class of planar graphs without $k_i$-cycles for $i=1,\ldots, q$. For example, $\mathcal{P}_{3,4,5}$ denotes planar graphs with girth $6$.
	
	It was shown in \cite{Zhu2017} that $4 +\frac{2}{9} \leq ch^s_f(\mathcal{P}) \leq 5$. The following result improves the lower bound. 
	
	\begin{proposition}
		\label{pro-planar}
		For each positive integer $m$, there is a planar graph $G$ which is not $(4m+
		\lfloor\frac{m-1}{3}\rfloor,m)$-choosable. Consequently, $ch^s_f(\mathcal{P}) \geq 4+\frac{1}{3}$.
	\end{proposition}
	
	\begin{proof}
		Let $T$ be the graph  as shown in Fig. \ref{T-planar graph}, $\epsilon$ be a real number such that $\epsilon m = \lfloor\frac{m-1}{3}\rfloor$. 
		Assume $A,B,C,D,E,F$ are pairwise disjoint sets of colours with $|A|=|B|=|C|=|D|=m$, $|E|=\epsilon m$ and $|F|=2m$.
		\begin{itemize}
			\item $L(u)=A$ and $L(v)=B$.
			\item $L(x) = L(y) = A\cup B \cup F \cup E$.
			\item $L(u_1) =  A\cup C \cup F \cup E$ and $L(v_1) = B \cup C \cup F \cup E$.
			\item $L(u_2) = A \cup D \cup F \cup E$. and $L(v_2) = B \cup D \cup F \cup E$.
			\item $L(z)=A \cup B \cup C \cup D \cup E$.
		\end{itemize}   
		\begin{figure}[H]
			\centering 
			\begin{tikzpicture}[>=latex,	
				roundnode/.style={circle, draw=black,fill= red, minimum size=1mm, inner sep=0pt}]  
				\node [roundnode] (u) at (0,2){}; 
				\node [roundnode] (z) at (0,0){};
				\node [roundnode] (u2) at (1.2,1.6/3){};	
				\node [roundnode] (u1) at (-1.2,1.6/3){};
				\node [roundnode] (v2) at (1.2,-1.6/3){};	
				\node [roundnode] (v1) at (-1.2,-1.6/3){};
				\node [roundnode] (v) at (0,-2){};
				\node [roundnode] (x) at (-8/3,0){};
				\node [roundnode] (y) at (8/3,0){};	
				
				\node at (0,2.2){$u$};
				\node at (0,-2.2){$v$};
				
				\node at (-0.85, -0.6){$v_1$}; 
				\node at (0.85, -0.6){$v_2$}; 
				\node at (-0.85,0.55){$u_1$};
				\node at (0.85, 0.6){$u_2$}; 
				
				\node at (0.15, 0.2){$z$}; 
				
				\node at (-3, 0){$x$}; 
				\node at (3, 0){$y$}; 
				
				\draw  (u)--(x)--(u1)--(z)--(u2)--(y)--(u)--(u2);
				\draw  (v1)--(u1)--(u)--(z);
				\draw  (u)--(z)--(v);
				\draw  (v)--(x)--(v1)--(z)--(v)--(v2)--(u2);
				\draw  (v1)--(v)--(y)--(v2); 
				\draw  (z)--(v2);
			\end{tikzpicture}
			\caption{The target graph $T$}
			\label{T-planar graph}
		\end{figure}
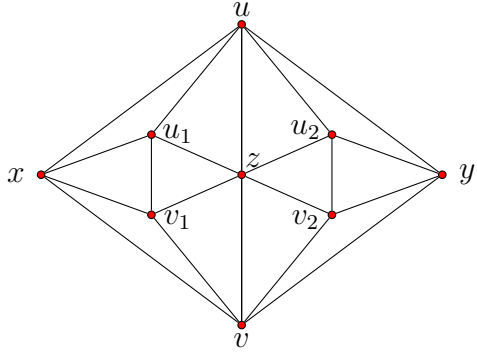
		
		Now we show that there is no $m$-fold $L$-colouring of $G$. Suppose to the contrary, $\phi$ is an $m$-fold $L$-colouring of $G$. Then $\phi(u)=A$ and $\phi(v)=B$. Note that $u_1v_1x$ is a clique, so each colour in $E \cup F$ can be used at most once in $u_1,v_1$ and $x$. As altogether, we use $3m$ distinct colours in these three vertices, at least $(1-\epsilon)m$ colours in $C$ are used on vertex $u_1$ and $v_1$, which implies that at most $\epsilon m$ colours in $C$ can be used at vertex $z$. By symmetric, at most $\epsilon m$ colours in $D$ can be used at vertex $z$. Recall that $|E|=\epsilon m$, so for the vertex $z$, 
		$$m=|\phi(z)|= |\phi(z) \cap C|+|\phi(z) \cap D| +|\phi(z)\cap E| \leq  3\epsilon m < m,$$
		a contradiction.

		Let $p = {(3+\epsilon)m \choose m}^2$. Let $G$ be obtained from the disjoint union of $p$ copies of $T$, by identifying all the copies of $u$ into a single vertex, also named $u$, and identifying all the copies of $v$ into a single vertex named $v$. 
		Let $L$ be the $(3+\epsilon)m$-list assignment of $G$ defined as follows: Let $L(u)=X$ and $L(v)=Y$, where  $X,Y$ are two disjoint set of size $(3+\epsilon)m$. For each pair of $m$-sets $(A,B)$, where $A \subseteq X$ and $B \subseteq Y$, we associate a copy of $T_{A,B}$ of $T$ so that the lists of the vertices of this copy of $T_{A,B}$ is as given above. Then $G$ is not $m$-fold $L$-colourable, for otherwise, $u$ is coloured with an $m$-subset $A$ of $X$, $v$ is coloured with an $m$-subset $B$ of $Y$. However, by the argument above,    $T_{A,B}$ has no $m$-fold $L$-colouring. 
	\end{proof}

	Next we consider the family $\mathcal{P}_{4}$. 
	
	\begin{proposition}
		\label{pro-planar-c4-free}
		For each positive integer $m$, there is a planar graph $G$ without $4$-cycle,  which is not $(3m+\lfloor\frac{m-1}{2}\rfloor,m)$-choosable. Consequently,  $3+\frac{1}{2} \leq ch^s_f(\mathcal{P}_{4}) \leq 4$.
	\end{proposition}
	\begin{proof}
		Let $T$ be labeled as shown in Fig.\ref{T4-planar graph}, $\epsilon$ be a real number such that $\epsilon m = \lfloor\frac{m-1}{2}\rfloor$. For any disjoint sets $A$ and $B$, we define a list assignment $L_{A,B}$ (when $A,B$ are clear, and there is no confusion, we write $L$ in short) of $T$ as follows: 
		Assume $A,B,C,D,E$ are pairwise disjoint sets of colours with $|A|=|B|=|D|=m$, $|C|=2m$ and $|E|=\epsilon m$.
		\begin{itemize}
			\item $L(u)=A$ and $L(v)=B$.
			\item $L(u_1) = L(u_2) = A\cup C \cup E$.
			\item $L(v_1) = L(v_2) = B \cup C \cup E$.  
			\item $L(x) = L(y) = C \cup D \cup E$.
		\end{itemize}   
		
		\begin{figure}[H]
			\centering 
			\begin{tikzpicture}[>=latex,	
				roundnode/.style={circle, draw=black,fill= red, minimum size=1mm, inner sep=0pt}]  
				\node [roundnode] (u) at (0,2){}; 
				\node [roundnode] (y) at (0.4,0){};
				\node [roundnode] (u2) at (1.25,0.6){};	
				\node [roundnode] (u1) at (-1.25,0.6){};
				\node [roundnode] (v2) at (1.25,-0.6){};	
				\node [roundnode] (v1) at (-1.25,-0.6){};
				\node [roundnode] (v) at (0,-2){};
				\node [roundnode] (x) at (-0.4,0){}; 
				\node at (0,2.2){$u$};
				\node at (0,-2.2){$v$};	
				\node at (-1.6, -0.6){$v_1$}; 
				\node at (1.6, -0.6){$v_2$}; 
				\node at (-1.6,0.6){$u_1$};
				\node at (1.6, 0.6){$u_2$}; 
				\node at (-0.35, 0.25){$x$}; 		
				\node at (0.35, 0.25){$y$}; 	
				\draw  (u)--(u1)--(x)--(y)--(u2)--(u);
				\draw (x)--(v1)--(v)--(v2)--(y); 
				\draw  (u1)--(v1);
				\draw  (u2)--(v2);
			\end{tikzpicture}
			\caption{For $\mathcal{P}_{4}$}
			\label{T4-planar graph}
		\end{figure}
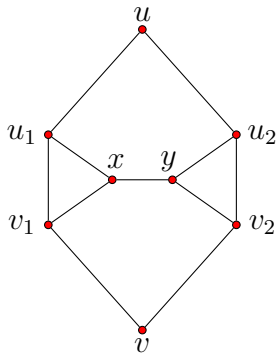
		
		By the same argument as in the proof of Theorem \ref{T-planar graph}, 
		it suffices to show that there is no $m$-fold $L$-colouring of $G$. Suppose to the contrary, $\phi$ is an $m$-fold $L$-colouring of $G$. Then $\phi(u)=A$ and $\phi(v)=B$. Note that $u_1v_1x$ is a clique, and we use $2m$ colours in $C\cup E$ on $u_1$ and $v_1$. Therefore, only $\epsilon m$ colours in $C \cup E$ can be used at $u_1$. By symmetric, only $\epsilon m$ colours in $C \cup E$ can be used at $y$. Note that $|D|=m$, so we have  
		$$2m=|\phi(x)|+|\phi(y)| \leq 2\epsilon m+|D|= 2\epsilon m+ m < 2m,$$
		a contradiction.
		
		It was proved in \cite{LSX2001} that  planar graphs without $4$-cycles are $(4m,m)$-choosable for any positive integer $m$. So $ch_f^s(\mathcal{P}_4) \le 4$. 
	\end{proof}
	
	Observe that $K_4$ is a planar graph without $k$-cycle for any $k \geq 5$, and $ch^s_f(K_4)=4$.  On the other hand,  it was shown in \cite{LSX2001} that every graph without $k$-cycle is $(4m,m)$-choosable, where $k \in \{4,5,6\}$. We have the following.
	
	\begin{observation}
		For any $k \geq 5$, $ch^s_f(\mathcal{P}_{k}) \geq 4$. In particular,
		$ch^s_f(\mathcal{P}_k) = 4$ when 
		$k \in \{5,6\}$.
	\end{observation}
	
	The construction in Proposition \ref{pro-planar} does not contain $k$-cycle for $k \geq 17$, which means that $ch^s_f(\mathcal{P}_{k}) > 4+ \frac{1}{3}$ for $k \ge 17$. It remains an open question as what is the smallest $k$  such that $ch^s_f(\mathcal{P}_{ k}) > 4$ ?
	
	The family of planar graphs   without $4$-, $5$-cycles has been studied extensively in the literature,   because of the well-known Steinberg's Conjecture, see \cite{Steinberg1993}. The conjecture asserts that   every planar graph contains neither $4$-cycle nor $5$-cycle is $3$-colourable. This conjecture was disproved \cite{CHKLS2017}. The list version of this conjecture was disproved earlier by Voigt \cite{Voigt2007}, who first constructed a non-$3$-choosable planar graph without $4$-, $5$-cycle with $344$ vertices. A  smaller one was given by Montassier \cite{Montassier2006} with $209$ vertices. 
	
	However, the counterexample graph to Steinberg's Conjecture given in \cite{CHKLS2017} is $(6,2)$-colourable, see Appendix \ref{Steinberg-example},  hence it has fractional chromatic number exactly $3$ (The graph contains a triangle, so the lower bound is $3$). 
	Therefore, it is $(3m,m)$-choosable for some $m$ by the main result in \cite{ATV1997}.  
	On the other hand, it is easily to check that all the non-3-choosable examples constructed in \cite{Voigt2007,Montassier2006,MRW2006,WWW2008} mentioned above are $3$-colourable, hence they are also $(3m,m)$-choosable for some $m$. So before the present paper,  it was unknown whether or not for every positive integer $m$, there is a   planar graph without $4$- and $5$-cycles which is not $(3m,m)$-choosable. In the following,  for each positive integer $m$, we construct a planar graph without cycles of length $4$ and $5$ which is not  $(3m + \lfloor\frac{m-1}{12}\rfloor,m)$-choosable. When $m=1$,  the graph has $164$ vertices, which is smaller than the counterexample graph found by Montassier in \cite{Montassier2006}.
	
	\begin{proposition}
		\label{pro-planar-C4&C5-free}
		For each positive integer $m$, there is a planar graph $G$ without $4$-cycle and $5$-cycle,  which is not $(3m+\lfloor\frac{m-1}{12}\rfloor,m)$-choosable. Consequently,  $ch^s_f(\mathcal{P}_{  4,5}) \geq 3+\frac{1}{12}$.
	\end{proposition}
	\begin{proof}
		Let $T$ be labeled as shown in Fig.\ref{T45-planar graph}, $\epsilon$ be a real number such that $\epsilon m = \lfloor\frac{m-1}{12}\rfloor$. 
		Assume $A,B,C,D,E$ are pairwise disjoint sets of colours with $|A|=|B|=|C|=m$, $|D|=2m$ and $|E|=\epsilon m$.
		\begin{itemize}
			\item $L(u)=A$, $L(v)=B$ and $L(w) = C \cup D \cup E$. 
			\item $L(u_1) = L(u_2) = L(w_1) = L(x_3) = L(y_1) = L(y_3) = L(z_1)= L(z_2) = A\cup D \cup E$.
			\item $L(v_1) = L(v_2) = L(w_2) = L(x_1) =L(x_2) = L(y_2) =L(y_3) =  B \cup D \cup E$.  
			\item $L(x) = L(y) = L(z) = A \cup B \cup  C \cup E$.
		\end{itemize}   
		
		\begin{figure}[h]
			\centering 
			\begin{tikzpicture}[>=latex,	
				roundnode/.style={circle, draw=black,fill= red, minimum size=1mm, inner sep=0pt}]  
				\node [roundnode] (u) at (0,3){}; 
				\node [roundnode] (z) at (0,2){}; 
				\node [roundnode] (z2) at (1.3,1.2){};
				\node [roundnode] (y) at (0.7,0.8){};  
				\node [roundnode] (y1) at (1,-0.5){}; 
				\node [roundnode] (y2) at (1,-1.5){};	
				\node [roundnode] (y3) at (0.6,-1){};
				\node [roundnode] (u2) at (2,1.8){};	
				\node [roundnode] (u1) at (-2,1.8){};
				\node [roundnode] (v2) at (2,-1.8){};	
				\node [roundnode] (v1) at (-2,-1.8){};
				\node [roundnode] (v) at (0,-3){};
				\node [roundnode] (z1) at (-1.3,1.2){}; 
				\node [roundnode] (x) at (-0.7,0.8){}; 
				\node [roundnode] (x1) at (-1,-0.5){}; 
				\node [roundnode] (x2) at (-1,-1.5){};	
				\node [roundnode] (x3) at (-0.6,-1){};
				\node [roundnode] (w) at (0,0.7){}; 
				\node [roundnode] (w1) at (-0.3,-0.4){}; 
				\node [roundnode] (w2) at (0.3,-0.4){}; 
				\node at (0,3.2){$u$};
				\node at (0,-3.2){$v$};
				\node at (0,2.2){$z$};
				\node at (-0.72,1.05){$x$};
				\node at (-1.25,-0.5){$x_1$};
				\node at (-1.25,-1.5){$x_2$};
				\node at (-0.4,-1.2){$x_3$};
				\node at (0.75,1.05){$y$};
				\node at (1.25,-0.5){$y_1$};
				\node at (1.25,-1.5){$y_2$};
				\node at (0.4,-1.2){$y_3$};
				\node at (0.2,0.7){$w$};
				\node at (-0.53,-0.25){$w_1$};
				\node at (0.53,-0.25){$w_2$};	
				
				\node at (-2.2, -2){$v_1$}; 
				\node at (2.2, -2){$v_2$}; 
				\node at (-2.2,2){$u_1$};
				\node at (2.2, 2){$u_2$}; 
				\node at (-1.6, 1.1){$z_1$}; 		
				\node at (1.65, 1.1){$z_2$}; 	
				\draw  (u)--(u1)--(z1)--(z)--(z2)--(u2)--(u);
				\draw (v1)--(v)--(v2)--(z2); 
				\draw (z)--(x)--(z1); 
				\draw (z)--(y)--(z2);
				\draw (x)--(x1)--(x2)--(x3)--(x1); 
				\draw (y)--(y1)--(y2)--(y3)--(y1); 
				\draw  (u1)--(v1)--(z1);
				\draw  (u2)--(v2);
				
				\draw (x2)--(v)--(y2);	
				\draw (x3)--(w1);
				\draw (y3)--(w2);
				\draw (z)--(w)--(w1)--(w2)--(w);
			\end{tikzpicture}
			\caption{Target graph for $\mathcal{P}_{4,5}$}
			\label{T45-planar graph}
		\end{figure}
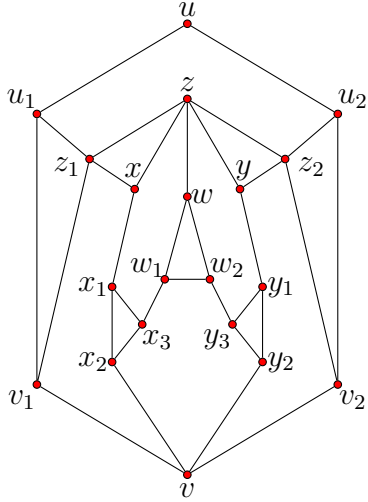 
		
		By the same argument as in the proof of Theorem \ref{T-planar graph}, it suffices to show that there is no $m$-fold $L$-colouring of $G$. Suppose to the contrary, $\phi$ is an $m$-fold $L$-colouring of $G$. Then $\phi(u)=A$ and $\phi(v)=B$. Note that $u_1v_1z_1$ is a clique, so each colour in $D \cup E$ can be used only once in these vertices, which means that  $$|\phi(z_1) \cap A| \geq 3m-|D \cup E|=m-\epsilon m.$$
		Similarly, $|\phi(z_2) \cap B| \geq m-\epsilon m$. So $|\phi(z) \cap A| \leq \epsilon m$ and $|\phi(z) \cap B| \leq \epsilon m$. We assume that $|\phi(z) \cap A| = \alpha m$, $|\phi(z) \cap B|=\beta m$ and $|\phi(z) \cap E|= \gamma m$, it is clear that $\alpha, \beta, \gamma \leq \epsilon$, and $|\phi(z) \cap C|= (1-(\alpha+\beta+\gamma))m$.
		
		As $zz_1x$ is a clique, each colour in $A$ can be used once in these three vertices, so 
		$$|\phi(x) \cap A| \leq |A|-|\phi(z_1) \cap A|-|\phi(z) \cap A| \leq (\epsilon-\alpha)m.$$
		Similarly, by considering the edge $xz$, we have that $|\phi(x) \cap C| \leq (\alpha+\beta +\gamma)m$ and $|\phi(x) \cap E| \leq (\epsilon-\gamma)m$. Note that  $\phi(z_1)\cap E$ might be empty. So we only have $|\phi(x) \cap E| \le |E| - |E \cap \phi(z)|$.   Therefore, we have 
		$$|\phi(x)\cap B| \geq m- |\phi(x) \cap A|-|\phi(x) \cap C|-|\phi(x) \cap E| \geq m-(2\epsilon +\beta)m.$$
		This implies that $|\phi(x_1) \cap B| \leq (2\epsilon+\beta)m$.
		
		Since $x_1x_2x_3$ is a clique, each colour in $D \cup E$ can be used at most once on these three vertices, but we need $3m$ colours for these vertices, so 
		$$|\phi(x_3) \cap A| \geq 3m-|\phi(x_1) \cap B|-|(\phi(x_1)\cup \phi(x_2)\cup \phi(x_3))\cap (D \cup E)|\geq m-(3\epsilon +\beta)m.$$
		Hence, $|\phi(w_1) \cap A| \leq (3\epsilon +\beta)m$.
		
		By symmetry, $|\phi(w_2) \cap A| \leq (3\epsilon +\alpha)m$. On the other hand, $|\phi(w)\cap C| \leq m-|\phi(z) \cap C| \leq (\alpha+\beta+\gamma)m$, so we have 
		\begin{align*}
			3m &= |\phi(w)|+|\phi(w_1)|+|\phi(w_2)| \\  
			& \leq |\phi(w_1) \cap A|+|\phi(w_2) \cap A| + |\phi(w)\cap C| + |(\phi(w)\cup \phi(w_1)\cup \phi(w_2))\cap (D \cup E)| \\
			& \leq  2m + 7\epsilon m + 2(\alpha+\beta)m + \gamma m \\
			& \leq 2m + 12\epsilon m  < 3m,
		\end{align*}
		a contradiction. 
	\end{proof}

	It was proved in  \cite{LZ2020} that the strong fractional choice number of $K_4$-minor-free graphs with girth at least $g$  is $2+\frac{1}{\lfloor(g+1)/4\rfloor}$.  
	Thus   the strong fractional choice number of the family of planar graphs of girth $5$ or $6$ is at least $3$, i.e., $ch^s_f(\mathcal{P}_{3,4}) \geq 3$. On the other hand,  extending the proofs in \cite{Thomassen-3-list,Thomassen-3-list-short}, Voigt \cite{Voigt1998} proved that every planar graphs with girth $5$ is $(3m,m)$-choosable, so the family of planar graphs of girth $5$ or $6$ has strong fractional choice number at most $3$.
	
	\begin{proposition}
		$ch^s_f(\mathcal{P}_{3,4}) = ch^s_f(\mathcal{P}_{3,4,5})=3$.
	\end{proposition}
	
	For the case of $\mathcal{P}_{3}$, the best known upper and lower bounds for their strong fractional chromatic number was obtained \cite{JZ2019}: $3 +\frac{1}{17} \leq ch^s_f(\mathcal{P}_3) \leq 4$. 
	
	\section{Open problems}
	\label{open problem}
	
	One basic unsolved problem concerning the strong fractional choice number is whether every rational $r \ge 2$ is the strong fractional choice number of a graph. We conjecture an affirmative answer.
	
	\begin{conjecture}
		For any rational number $r \geq 2$, there exists a graph $G$ such that $ch^s_f(G)=r$ and   a graph $G'$ with $\chi^s_{f,P}(G')=r$.
	\end{conjecture}
	
	Erd\H{o}s, Rubin and Taylor \cite{ERT1979} characterized all the $2$-choosable graphs. However, it seems to be a difficult problem to characterize all graphs $G$ with $ch_f^s(G)=2$. In a companion paper \cite{XZ2021+}, we proved that every $3$-choice critical bipartite graph $G$ (i.e., $G$ is not $2$-choosable, but every proper subgraph of $G$ is $2$-choosable) has strong fractioal choice number $2$. 
	
	\begin{question}
		Given a characterization of the class of graphs whose strong fractional choice number are $2$.
	\end{question}

	It was asked by Erd\H{o}s, Rubin and Taylor \cite{ERT1979} that whether every $(a,b)$-choosable graph is also $(am,bm)$-choosable for any positive integer $m$. The case $(a,b)=(2,1)$ was affirmed by Tuza and Voigt \cite{TV1996-2m}, but the case $a \geq 4$ and $b=1$ was negatived by Dvo\v{r}\'{a}k, Hu and Sereni \cite{DHS2019} recently. For a relax and possibly correct version, we ask the following question.
	
	\begin{question}
		\label{conj-ch}
		Is it true that $ch^s_{f}(G) \leq ch(G)$ for any graph $G$?
	\end{question}
	
	Similarly,  it was conjectured by Mahoney, Meng and Zhu \cite{MMZ2015} that every $(a,b)$-paintable graph is also $(am,bm)$-paintable for any positive integer $m$. We also ask the following weaker problem.
	
	\begin{question}
		\label{conj-chiP}
		Is it true that  $\chi_{f,P}^s(G) \le \chi_P(G)$  for any graph $G$?
	\end{question} 
	
	Planar graph colouring is a central problem with respect to many colouring concepts.  This is also the case for the strong fractional choice number of graphs. 
	
	\begin{question}
		What is the exact value of $ch^s_f(\mathcal{P})$?  Is it true that $ch^s_f(\mathcal{P}) < 5$?
	\end{question}

	\begin{question}
		What is the exact value of $ch^s_f(\mathcal{P}_3)$?  Is it true that $ch^s_f(\mathcal{P}_3) < 4$? 
	\end{question}
	
	\begin{question}
		What is the exact value of $ch^s_f(\mathcal{P}_{4,5})$?  Is it true that $ch^s_f(\mathcal{P}_{4,5}) < 4$? 
	\end{question}
	
	Although Steinberg's conjecture is false, the fractional chromatic number and the strong fractional chromatic number of graphs in $\mathcal{P}_{4,5}$ is open. It was proved in \cite{DH2019} that  for any $G \in \mathcal{P}_{4,5}$, $\chi_f(G) \leq 11/3$.  The following question remains open.

		\begin{question}
			\label{steinberg-type}
			Is it true that every graph  $G \in \mathcal{P}_{4,5}$ has $\chi_f(G) \le  3$, or even has $\chi_f^s(G) \le 3$?
	\end{question}

	\bibliographystyle{abbrv} 
	\bibliography{reference}

	\begin{appendices}
		\section{}
		\label{Steinberg-example}
		In this part, we give a $(6,2)$-colouring $\phi$ of the counterexample to Steinberg's conjecture presented in \cite{CHKLS2017}. The counterexample constructed in \cite{CHKLS2017} is the graph depicted in Figure 5, where Figure 6 depicts two copies of $G_2$.  We first pre-colour part of the graph in Fig.3. as follows: $\phi(a)=\{1,3\}$, $\phi(b)=\{5,6\}$, $\phi(c)=\{1,2\}$, $\phi(d)=\{2,3\}$, $\phi(e)=\{4,5\}$, $\phi(f)=\{1,6\}$, $\phi(c')=\{3,4\}$, $\phi(d')=\{1,4\}$, $\phi(e')=\{2,5\}$ and $\phi(f')=\{3,6\}$.
		
		\vspace{-1.6cm}
		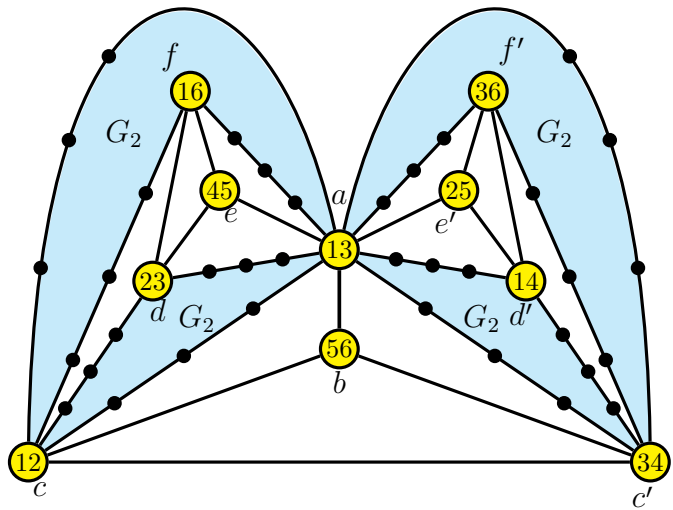
\begin{figure}[H] 
			\centering
			\begin{tikzpicture}[scale=0.88,>=latex,	
				roundnode/.style={circle, draw=black,fill=black, minimum size=1.5mm, inner sep=0pt},
				blanknode/.style={circle, draw = black, fill=yellow, very thick, minimum size=3mm, inner sep=0.3mm},]
				 
				\draw[draw=none,fill=cyan!20] (200:5)--(160:3)--(90:1.5)--(200:5);
				\draw[draw=none,fill=cyan!20] (340:5)--(20:3)--(90:1.5)--(340:5);
					 
				\draw[draw=none,fill=cyan!20] (200:5)--(120:4.5)--(90:1.5).. controls (100:7) and (125:8) .. (200:5);

				\draw[draw=none,fill=cyan!20] (340:5)--(60:4.5)--(90:1.5).. controls (80:7) and (55:8) .. (340:5);	
					 
				\node [blanknode] (c') at (340:5){\footnotesize $34$};
				\node [blanknode] (b) at (90:0){\footnotesize $56$};	
				\node [blanknode] (a) at (90:1.5){\footnotesize $13$};	
				\node [blanknode] (c) at (200:5){\footnotesize $12$}; 
				
				\node [blanknode] (d) at (160:3){\footnotesize $23$};	
				\node [blanknode] (d') at (20:3){\footnotesize $14$}; 
				
				\node [blanknode] (e) at (127:3){\footnotesize $45$};	
				\node [blanknode] (e') at (53:3){\footnotesize $25$}; 
				
				\node [blanknode] (f) at (120:4.5){\footnotesize $16$};	
				\node [blanknode] (f') at (60:4.5){\footnotesize $36$}; 
				
				\draw[line width = 1.2pt] (a)--(c)
				\foreach \t in {1/4,2/4,3/4}{pic [draw,pos=\t] {code={\node [roundnode]{};}}};
				
				\draw[line width = 1.2pt] (a)--(d)
				\foreach \t in {1/4,2/4,3/4}{pic [draw,pos=\t] {code={\node [roundnode]{};}}};
				 
				\draw[line width = 1.2pt] (c)--(d)
				\foreach \t in {1/4,2/4,3/4}{pic [draw,pos=\t] {code={\node [roundnode]{};}}};
				
				\draw[line width = 1.2pt] (a)--(c')
				\foreach \t in {1/4,2/4,3/4}{pic [draw,pos=\t] {code={\node [roundnode]{};}}};
				
				\draw[line width = 1.2pt] (a)--(d')
				\foreach \t in {1/4,2/4,3/4}{pic [draw,pos=\t] {code={\node [roundnode]{};}}};
				
				\draw[line width = 1.2pt] (c')--(d')
				\foreach \t in {1/4,2/4,3/4}{pic [draw,pos=\t] {code={\node [roundnode]{};}}};

				\draw[line width = 1.2pt] (a)--(f)
				\foreach \t in {1/4,2/4,3/4}{pic [draw,pos=\t] {code={\node [roundnode]{};}}};
				 
	 			\draw[line width = 1.2pt] (a)--(f')
	 			\foreach \t in {1/4,2/4,3/4}{pic [draw,pos=\t] {code={\node [roundnode]{};}}};
	 			
	 			\draw[line width = 1.2pt] (c)--(f)
	 			\foreach \t in {1/4,2/4,3/4}{pic [draw,pos=\t] {code={\node [roundnode]{};}}};
	 			
	 			\draw[line width = 1.2pt] (c')--(f')
		 		\foreach \t in {1/4,2/4,3/4}{pic [draw,pos=\t] {code={\node [roundnode]{};}}};

				\draw[line width = 1.2pt] (c)..controls (125:8) and (100:6.9) .. (a)
				\foreach \t in {1/8,2/8,3/8}{pic [draw,pos=\t] {code={\node [roundnode]{};}}};
				
				\draw[line width = 1.2pt] (c')..controls (55:8) and (80:6.9) .. (a)
				\foreach \t in {1/8,2/8,3/8}{pic [draw,pos=\t] {code={\node [roundnode]{};}}};
				
				\node at (334:5.1){$c'$};
				\node at (270:0.5){$b$};
				\node at (205:5){$c$}; 
				\node at (90:2.3){$a$};
				
				\node at (168:2.8){$d$};  
				\node at (168:2.2){$G_2$}; 
				
				\node at (11:2.8){$d'$};  
				\node at (12:2.2){$G_2$}; 
				
				\node at (129:2.6){$e$};  
				\node at (50:2.5){$e'$}; 
				
					\node at (120:5.1){$f$};  
				\node at (60:5.2){$f'$}; 
			 
			    \node at (45:4.6){$G_2$}; 
			    \node at (135:4.6){$G_2$};
			    	
			 	\draw[line width = 1.4pt] (a)--(b);
				\draw[line width = 1.2pt] (c)--(b)--(c')--(c);
				
				\draw[line width = 1.2pt] (a)--(e);
				\draw[line width = 1.2pt] (d)--(e)--(f)--(d);
				
				\draw[line width = 1.2pt] (a)--(e');
				\draw[line width = 1.2pt] (d')--(e')--(f')--(d');
 
			\end{tikzpicture}  
			\caption{The counterexample to Steinberg's Conjecture in \cite{CHKLS2017}}
		\end{figure}

We shall show that this partial colouring can be extended to a $2$-fold $6$-colouring of the whole graph. 
By symmetry, it suffices to extend the partial colouring to the left two copies of $G_2$, which is given in Figure \ref{G2coloring}.	
		 
		\vspace{-0.4cm}
		\begin{figure}[H]
			\centering 
			\begin{minipage}[t]{0.49\textwidth}
				\begin{tikzpicture}[scale=0.88,>=latex,	
					roundnode/.style={circle, draw = black, very thick, minimum size=3mm, inner sep=0.5pt},
					colorednode/.style={circle, draw = black, fill=yellow, very thick, minimum size=3mm, inner sep=1pt}] 
					
					\node [colorednode] (X) at (90:4){\scriptsize$23$};
					\node [colorednode] (Y) at (210:4){\scriptsize$12$};
					\node [colorednode] (Z) at (330:4){\scriptsize$13$};
					
					\node [roundnode] (A1) at (30:0.8){\scriptsize$24$};
					\node [roundnode] (A2) at (150:0.8){\scriptsize$13$};
					\node [roundnode] (A3) at (270:0.8){\scriptsize$56$};
					
					\node [roundnode] (B1) at (70:1.5){\scriptsize$13$};
					\node [roundnode] (B2) at (110:1.5){\scriptsize$46$};
					\node [roundnode] (B3) at (190:1.5){\scriptsize$25$};
					\node [roundnode] (B4) at (230:1.5){\scriptsize$24$};
					\node [roundnode] (B5) at (310:1.5){\scriptsize$13$};
					\node [roundnode] (B6) at (350:1.5){\scriptsize$56$};
					
					\node [roundnode] (C1) at (83:2.5){\scriptsize$45$};
					\node [roundnode] (C2) at (97:2.5){\scriptsize$15$};
					\node [roundnode] (C3) at (203:2.5){\scriptsize$34$};
					\node [roundnode] (C4) at (217:2.5){\scriptsize$35$};
					\node [roundnode] (C5) at (323:2.5){\scriptsize$56$};
					\node [roundnode] (C6) at (337:2.5){\scriptsize$24$};				  
					
					\node [roundnode] (D1) at (62:2.1){\scriptsize$26$};
					\node [roundnode] (D2) at (118:2.1){\scriptsize$23$};
					\node [roundnode] (D3) at (182:2.1){\scriptsize$16$};
					\node [roundnode] (D4) at (238:2.1){\scriptsize $16$};
					\node [roundnode] (D5) at (302:2.1){\scriptsize $24$};
					\node [roundnode] (D6) at (358:2.1){\scriptsize $13$};
					
					\node [roundnode] (E1) at (43:1.6){\scriptsize $35$};
					\node [roundnode] (E2) at (137:1.6){\scriptsize $56$};
					\node [roundnode] (E3) at (163:1.6){\scriptsize $23$};
					\node [roundnode] (E4) at (257:1.6){\scriptsize $23$};
					\node [roundnode] (E5) at (283:1.6){\scriptsize $16$};
					\node [roundnode] (E6) at (17:1.6){\scriptsize $26$};
					
					\node [roundnode] (F1) at (30:2.2){\scriptsize $14$};
					\node [roundnode] (F2) at (150:2.2){\scriptsize $14$};
					\node [roundnode] (F3) at (270:2.2){\scriptsize $45$};
					
					\node [roundnode] (G1) at (30:2.8){\scriptsize $23$};
					\node [roundnode] (G2) at (150:2.8){\scriptsize$23$};
					\node [roundnode] (G3) at (270:2.8){\scriptsize$13$};
					
					\node [roundnode] (H1) at (62:3){\scriptsize$16$};
					\node [roundnode] (H2) at (118:3){\scriptsize$46$};
					\node [roundnode] (H3) at (182:3){\scriptsize$56$};
					\node [roundnode] (H4) at (238:3){\scriptsize $46$};
					\node [roundnode] (H5) at (302:3){\scriptsize $24$};
					\node [roundnode] (H6) at (358:3){\scriptsize $56$};
					
					\draw[line width = 1.2pt] (A1)--(A2)--(A3)--(A1);
					\draw[line width = 1.2pt] (B1)--(B6)--(A1); 
					\draw[line width = 1.2pt] (B2)--(B3);  
					\draw[line width = 1.2pt] (B4)--(B5); 
					
					\draw[line width = 1.2pt] (A1)--(B1)--(C1)--(X);
					\draw[line width = 1.2pt] (A2)--(B2)--(C2)--(X);
					
					\draw[line width = 1.2pt] (A2)--(B3)--(C3)--(Y);
					\draw[line width = 1.2pt] (A3)--(B4)--(C4)--(Y);
					
					\draw[line width = 1.2pt] (A3)--(B5)--(C5)--(Z);
					\draw[line width = 1.2pt] (A1)--(B6)--(C6)--(Z);
					
					\draw[line width = 1.2pt] (C1)--(D1)--(B1); 
					\draw[line width = 1.2pt] (C2)--(D2)--(B2);
					\draw[line width = 1.2pt] (C3)--(D3)--(B3);
					\draw[line width = 1.2pt] (C4)--(D4)--(B4);
					\draw[line width = 1.2pt] (C5)--(D5)--(B5);
					\draw[line width = 1.2pt] (C6)--(D6)--(B6);

					\draw[line width = 1.2pt] (D1)--(E1)--(E6)--(D6);			
					\draw[line width = 1.2pt] (D2)--(E2)--(E3)--(D3);
					\draw[line width = 1.2pt] (D4)--(E4)--(E5)--(D5);
					
					\draw[line width = 1.2pt] (E1)--(F1)--(E6); 
					\draw[line width = 1.2pt] (E2)--(F2)--(E3);
					\draw[line width = 1.2pt] (E4)--(F3)--(E5);
					
					\draw[line width = 1.2pt] (G1)--(F1); 
					\draw[line width = 1.2pt] (G2)--(F2);
					\draw[line width = 1.2pt] (G3)--(F3);
					
					\draw[line width = 1.2pt] (C1)--(H1)--(X); 
					\draw[line width = 1.2pt] (C2)--(H2)--(X);  
					\draw[line width = 1.2pt] (C3)--(H3)--(Y); 
					\draw[line width = 1.2pt] (C4)--(H4)--(Y);  
					\draw[line width = 1.2pt] (C5)--(H5)--(Z); 
					\draw[line width = 1.2pt] (C6)--(H6)--(Z); 
					
					\draw[line width = 1.2pt] (H1)--(G1)--(H6); 
					\draw[line width = 1.2pt] (H2)--(G2)--(H3);
					\draw[line width = 1.2pt] (H4)--(G3)--(H5);	
					
					\node at  (90:4.5){$d$};
					\node at  (214:4.3){$c$};
					\node at  (326:4.3){$a$};
				\end{tikzpicture} 
			\end{minipage}
			\begin{minipage}[t]{0.49\textwidth}
				\begin{tikzpicture}[scale=0.9,>=latex,	
					roundnode/.style={circle, draw = black, very thick, minimum size=3mm, inner sep=1pt},
					colorednode/.style={circle, draw = black, fill=yellow, very thick, minimum size=3mm, inner sep=1pt}] 
					
					\node [colorednode] (X) at (90:4){\scriptsize$16$};
					\node [colorednode] (Y) at (210:4){\scriptsize$12$};
					\node [colorednode] (Z) at (330:4){\scriptsize$13$};
					
					\node [roundnode] (A1) at (30:0.8){\scriptsize$24$};
					\node [roundnode] (A2) at (150:0.8){\scriptsize$13$};
					\node [roundnode] (A3) at (270:0.8){\scriptsize$56$};
					
					\node [roundnode] (B1) at (70:1.5){\scriptsize$13$};
					\node [roundnode] (B2) at (110:1.5){\scriptsize$46$};
					\node [roundnode] (B3) at (190:1.5){\scriptsize$25$};
					\node [roundnode] (B4) at (230:1.5){\scriptsize$24$};
					\node [roundnode] (B5) at (310:1.5){\scriptsize$13$};
					\node [roundnode] (B6) at (350:1.5){\scriptsize$56$};
					
					\node [roundnode] (C1) at (83:2.5){\scriptsize$45$};
					\node [roundnode] (C2) at (97:2.5){\scriptsize$35$};
					\node [roundnode] (C3) at (203:2.5){\scriptsize$36$};
					\node [roundnode] (C4) at (217:2.5){\scriptsize$35$};
					\node [roundnode] (C5) at (323:2.5){\scriptsize$56$};
					\node [roundnode] (C6) at (337:2.5){\scriptsize$24$};				  
					
					\node [roundnode] (D1) at (62:2.1){\scriptsize$26$};
					\node [roundnode] (D2) at (118:2.1){\scriptsize$12$};
					\node [roundnode] (D3) at (182:2.1){\scriptsize$14$};
					\node [roundnode] (D4) at (238:2.1){\scriptsize $16$};
					\node [roundnode] (D5) at (302:2.1){\scriptsize $24$};
					\node [roundnode] (D6) at (358:2.1){\scriptsize $13$};
					
					\node [roundnode] (E1) at (43:1.6){\scriptsize $14$};
					\node [roundnode] (E2) at (137:1.6){\scriptsize $16$};
					\node [roundnode] (E3) at (163:1.6){\scriptsize $23$};
					\node [roundnode] (E4) at (257:1.6){\scriptsize $23$};
					\node [roundnode] (E5) at (283:1.6){\scriptsize $16$};
					\node [roundnode] (E6) at (17:1.6){\scriptsize $26$};
					
					\node [roundnode] (F1) at (30:2.2){\scriptsize $35$};
					\node [roundnode] (F2) at (150:2.2){\scriptsize $45$};
					\node [roundnode] (F3) at (270:2.2){\scriptsize $45$};
					
					\node [roundnode] (G1) at (30:2.8){\scriptsize $14$};
					\node [roundnode] (G2) at (150:2.8){\scriptsize$13$};
					\node [roundnode] (G3) at (270:2.8){\scriptsize$13$};
					
					\node [roundnode] (H1) at (62:3){\scriptsize$23$};
					\node [roundnode] (H2) at (118:3){\scriptsize$24$};
					\node [roundnode] (H3) at (182:3){\scriptsize$45$};
					\node [roundnode] (H4) at (238:3){\scriptsize $46$};
					\node [roundnode] (H5) at (302:3){\scriptsize $24$};
					\node [roundnode] (H6) at (358:3){\scriptsize $56$};
					
					\draw[line width = 1.2pt] (A1)--(A2)--(A3)--(A1);
					\draw[line width = 1.2pt] (B1)--(B6)--(A1); 
					\draw[line width = 1.2pt] (B2)--(B3);  
					\draw[line width = 1.2pt] (B4)--(B5); 
					
					\draw[line width = 1.2pt] (A1)--(B1)--(C1)--(X);
					\draw[line width = 1.2pt] (A2)--(B2)--(C2)--(X);
					
					\draw[line width = 1.2pt] (A2)--(B3)--(C3)--(Y);
					\draw[line width = 1.2pt] (A3)--(B4)--(C4)--(Y);
					
					\draw[line width = 1.2pt] (A3)--(B5)--(C5)--(Z);
					\draw[line width = 1.2pt] (A1)--(B6)--(C6)--(Z);
					
					\draw[line width = 1.2pt] (C1)--(D1)--(B1); 
					\draw[line width = 1.2pt] (C2)--(D2)--(B2);
					\draw[line width = 1.2pt] (C3)--(D3)--(B3);
					\draw[line width = 1.2pt] (C4)--(D4)--(B4);
					\draw[line width = 1.2pt] (C5)--(D5)--(B5);
					\draw[line width = 1.2pt] (C6)--(D6)--(B6);

					\draw[line width = 1.2pt] (D1)--(E1)--(E6)--(D6);			
					\draw[line width = 1.2pt] (D2)--(E2)--(E3)--(D3);
					\draw[line width = 1.2pt] (D4)--(E4)--(E5)--(D5);
					
					\draw[line width = 1.2pt] (E1)--(F1)--(E6); 
					\draw[line width = 1.2pt] (E2)--(F2)--(E3);
					\draw[line width = 1.2pt] (E4)--(F3)--(E5);
					
					\draw[line width = 1.2pt] (G1)--(F1); 
					\draw[line width = 1.2pt] (G2)--(F2);
					\draw[line width = 1.2pt] (G3)--(F3);
					
					\draw[line width = 1.2pt] (C1)--(H1)--(X); 
					\draw[line width = 1.2pt] (C2)--(H2)--(X);  
					\draw[line width = 1.2pt] (C3)--(H3)--(Y); 
					\draw[line width = 1.2pt] (C4)--(H4)--(Y);  
					\draw[line width = 1.2pt] (C5)--(H5)--(Z); 
					\draw[line width = 1.2pt] (C6)--(H6)--(Z); 
					
					\draw[line width = 1.2pt] (H1)--(G1)--(H6); 
					\draw[line width = 1.2pt] (H2)--(G2)--(H3);
					\draw[line width = 1.2pt] (H4)--(G3)--(H5); 
					
					\node at  (90:4.5){$f$};
					\node at  (215:4.3){$c$};
					\node at  (325:4.3){$a$};
				\end{tikzpicture}
			\end{minipage}
			\caption{$(6,2)$-colourings of the left two copies of $G_2$}
			\label{G2coloring}
		\end{figure}
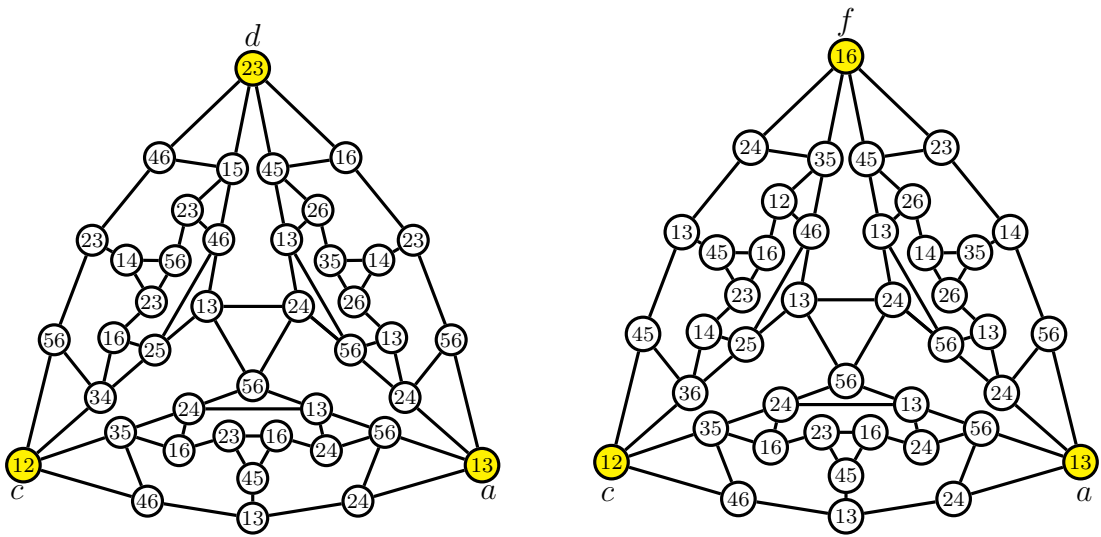

	\end{appendices}

\end{document}